\documentclass[a4paper,11pt]{article}
\usepackage{amsfonts}
\usepackage[T1]{fontenc}
\usepackage{microtype}
\usepackage{makeidx}
\usepackage{setspace}
\usepackage{authblk}
\usepackage{epigraph}
\usepackage{graphicx}
\usepackage{xcolor}
\usepackage[all]{xy}
\usepackage{amsmath}
\usepackage{amssymb}
\usepackage{booktabs}
\usepackage{braket}
\usepackage{array}
\usepackage{tabularx}
\usepackage{multirow}
\usepackage{indentfirst}
\usepackage{cite}
\usepackage{stmaryrd}
\usepackage{faktor}
\usepackage{titling}
\usepackage{tikz}
\usepackage{dsfont}
\usepackage{url}
\usepackage{hyperref}
\usepackage{tikz-cd}
\usepackage{tkz-graph}
\usepackage{caption}

\usetikzlibrary{matrix,arrows,decorations.pathmorphing}
    \renewcommand{\leq}{\leqslant}
    \renewcommand{\geq}{\geqslant}
\usepackage{amsthm}
\theoremstyle{plain}
\newtheorem{thm}{Theorem}[section]
\newtheorem{dfn}[thm]{Definition}

\newtheorem{prop}[thm]{Proposition}

\newtheorem{que}[thm]{Question}

\theoremstyle{definition}
\newtheorem{ex}[thm]{Example}

\theoremstyle{remark}
\newtheorem{oss}[thm]{Remark}

\DeclareMathOperator{\U}{\mathrm{U}}

\DeclareMathOperator{\id}{\mathrm{Id}}

\DeclareMathOperator{\aut}{\mathrm{Aut}}
\DeclareMathOperator{\imm}{\mathrm{Im}}

\DeclareMathOperator{\Pal}{\mathrm{Pal}}

\DeclareMathOperator{\N}{\mathbb{N}}
\DeclareMathOperator{\Z}{\mathbb{Z}}

\DeclareMathOperator{\sort}{\mathrm{sorted}}

\title{A weak Lehmer code for type $F_4$}

\author{Paolo Sentinelli\thanks{ Dipartimento di Matematica, Politecnico di Milano, Milan, Italy. \\ \href{mailto:paolosentinelli@gmail.com}{paolo.sentinelli@protonmail.com}} \ and Andrea Zatti}
\date{}

\begin{document}

\maketitle

\begin{abstract}
We provide an algorithm to construct a multicomplex for any lower Bruhat interval of $F_4$, such that its rank--generating function equals that of the Bruhat interval. For Weyl groups, it is always possible to find such a multicomplex thanks to the work of Bj\"{o}rner and Ekedahl. The algorithm is based on only two functions, which weaken the notion of Lehmer code for finite Coxeter groups, motivated by the fact that a strong Lehmer code for type $F_4$ does not exist. We also realize the set of palindromic Poincaré polynomials of $F_4$ as an induced subposet of the Bruhat order that forms a lattice.   
 \end{abstract}

\section{Introduction}

In the last section of the paper \cite{Billey1}, the authors glimpse the possibility that any rank--symmetric lower Bruhat interval in a Weyl group admits a product of chains, of the same cardinality as a subposet. It is certainly true that there exists a bijective and rank--preserving function between any rank--symmetric lower Bruhat interval in a Weyl group and a product of chains; in fact it is known that the rank--generating function of such an interval is a product of $q$-analogues (see e.g. \cite{Carrell}). Sara Billey has proved in \cite{billei} that, for types $A_n$ and $B_n$, there exists a bijective function from any rank--symmetric lower Bruhat interval to a product of chains, whose inverse is order--preserving. The same result can be verified computationally in type $H_3$. This property does not necessarily hold for all Weyl groups. In particular, in the language of \cite{BS}, it is proved computationally in \cite{bishop} and in the present paper, with other techniques, that a Coxeter system $(W,S)$ of type $F_4$ does not admit a Lehmer code, i.e. there does not exist any order--preserving bijective function from $[2]\times [6] \times [8] \times [12]$ to $W$. 

A.  Bj\"{o}rner and T. Ekedahl showed in \cite[Theorem E]{BE} that the ranks of a lower Bruhat interval of a Weyl group form an $M$-sequence. This result implies that for any lower Bruhat interval of a Weyl group there exists a multicomplex whose $f$-polynomial equals the rank--generating function of the interval. For types $A_n$, $B_n$ and $D_n$, the Lehmer codes constructed in \cite{BS}  provide, for any lower Bruhat interval, an explicit multicomplex with this property. This is also obtained for a Coxeter group of type $H_3$, a case not contemplated by \cite[Theorem E]{BE}. In Section \ref{sezione weak}, we weaken the notion of Lehmer code to provide an algorithm which realizes, for any lower Bruhat interval in type $F_4$, a multicomplex with the same rank--generating function. We also introduce the notion of principal and unimodal elements, extending those of \cite[Section 6]{BS}, to present as a lattice the set of palindromic rank--generating functions of lower Bruhat intervals in type $F_4$ (see Figures \ref{Hasse0} and \ref{Hasse1}), i.e. the set of Poincaré polynomials for the cohomology of rationally smooth Schubert varieties  in type $F_4$. This presentation extends to type $F_4$ the results of \cite[Th. 7.11 and Ex. 6.8]{BS}, obtained for types $A_n$ and $H_3$, linking, by isomorphism, a natural order on the set of palindromic Poincaré polynomials with the Bruhat order on unimodal elements, which turns out to be a lattice (see Figure \ref{Hasse1}).

\section{Notation and preliminaries}

We refer to \cite[Section 2]{BS} for most notation and preliminaries. We only recall some of the fundamental ones. For $n\in \N$ we let $[n]:=\{1,2,\ldots,n\}$ and $[n]_0:=\{0,1,2,\ldots,n\}$. The $q$-\emph{analogue} of $n$ is the polynomial $[n]_q:=\sum_{i=0}^{n-1}q^i$. For an element $x\in \N^k$, we let $\sort(x) \in \N^k$ be the tuple $x$ increasingly ordered. The set $\N^k$ is considered to be ordered componentwise. 
Let $(P,\leq)$ be a poset. A  subset $J \subseteq P$ is an \emph{order ideal} if $x \in J$ and $y\leq x$ imply $y\in J$.

\begin{dfn} Let $k\in \N$. An order ideal of $\N^k$ is called \emph{multicomplex}.
\end{dfn}
Notice that an order ideal of $[1]_0^k$ is a simplicial complex. The \emph{$f$-vector} of a multicomplex $J$ is the sequence $(f_0,f_1,\ldots)$, where $f_i:=|\{x \in J: \rho(x)=i\}|$, for all $i\in \N$, where $\rho$ is the rank function. The \emph{$f$-polynomial} of a finite multicomplex $J$ is the rank--generating function $\sum_{x\in J}q^{\rho(x)}$.  

\begin{dfn} \label{def msequenza}
    A sequence $g$ of natural numbers is called \emph{M}-sequence if there exists a multicomplex whose $f$-vector is $g$.
\end{dfn}
We recall now some facts on Coxeter groups. Let $(W,S)$ be a Coxeter system, $e\in W$ the identity, $\ell : W\rightarrow \N \cup\{0\}$ the length function and $J \subseteq S$. Then for every $w\in W$ there exists a unique factorization \begin{equation} \label{fattorizzazione}
    w=w_Jw^J
\end{equation} such that $\ell(w)=\ell(w_J)+\ell(w^J)$ where $w_J$ is an element of the subgroup generated by $J$, and $\ell(sw^J)>\ell(w^J)$, for all $s\in J$. 
 We denote by $\leq$ the Bruhat order on $W$ and with the same notation the componentwise order on $\N^k$.
The group $W$ with the Bruhat order is graded with rank function $\ell$ and the polynomial $h_w:=\sum\limits_{v\leq w}q^{\ell(w)}$ is usually called the Poincaré polynomial of the interval $[e,w]$. For some special elements $w\in W$, such polynomials factorize nicely. For example, if $|W|<\infty$, $k:=|S|$ and $w_0$ is the element of maximal length, then $$h_{w_0}=\sum\limits_{w \leq w_0}q^{\ell(w)}=\prod\limits_{i=1}^k[e_i+1]_q.$$ 
The natural numbers $1=e_1\leq \ldots \leq e_k$ are called the \emph{exponents} of $(W,S)$. The exponents for type $F_4$ are $1,5,7,11$.
It is known that in a Weyl group, the Schubert variety corresponding to an element $w$ is rationally smooth if and only if the polynomial $h_w$ is palindromic (see \cite[Theorem E]{Carrell}).  

We end this section by recalling the definition of Lehmer code for a finite Coxeter system, as appears in \cite[Definition 4.2]{BS}.
\begin{dfn} \label{def codici di L}
    Let $(W,S)$ be a finite Coxeter system with exponents $e_1,\ldots,e_n$. A bijective function $L: W \rightarrow \prod_{i=1}^n\{0,\ldots,e_i\}$ is a \emph{Lehmer code} if
    $$L^{-1}: \prod\limits_{i=1}^n\{0,\ldots,e_i\}\rightarrow (W,\leq)$$
    is order preserving.
\end{dfn}
Notice that the function $L$ of Definition \ref{def codici di L} is rank--preserving. In this paper we refer to a Lehmer code as a \emph{strong Lehmer code}. 
We extend the notion of Lehmer code to infinite Coxeter systems of finite rank. \begin{dfn} \label{def codici di L 2}
    Let $(W,S)$ be an infinite Coxeter system such that $k:=|S|<\infty $. A rank--preserving injective function $L: W \rightarrow \N^k$ is a \emph{Lehmer code} if $\imm(L)$ is an order ideal of $\N^k$ and
    $L^{-1}: \imm(L) \rightarrow (W,\leq)$
    is order preserving.
\end{dfn}

\begin{ex}
Let $(W,S)$ be a Coxeter system of type $\tilde{A}_1$, where $S=\{s,t\}$. If $L : W \rightarrow \N^2$ is the function defined by $L(w)=(\ell(w_{\{s\}}),\ell(w^{\{s\}}))$, for all $w\in W$, then $L$ is a Lehmer code for $(W,S)$.
\end{ex}

Definition \ref{def codici di L 2} includes also cases of the following type.

\begin{ex} \label{esempio S3}
    Let $(W,S)$ of type $A_2$, $S=\{s,t\}$ and $X\subseteq \N^2$ the order ideal with maxima $\{(1,1),(3,0)\}$. Define a function $L: W \rightarrow X$ by setting $L(e)=(0,0)$, $L(s)=(1,0)$, $L(st)=(2,0)$, $L(sts)=(3,0)$, $L(t)=(0,1)$ and $L(ts)=(1,1)$. Then $L$ is a Lehmer code according to Definition \ref{def codici di L 2}.  
\end{ex}

\begin{figure}[htbp]
    \centering
    \begin{tikzpicture}[scale=1.5, every node/.style={inner sep=2pt}]
        
        \begin{scope}[xshift=0cm]
            \node (00) at (0,0) {$(0,0)$};
            \node (10) at (1,1) {$(1,0)$};
            
            \node (01) at (-1, 1) {$(0,1)$};
            \node (11) at (-1, 2) {$(1,1)$};
            
            \node (20) at (1, 2) {$(2,0)$};
            \node (30) at (0, 3) {$(3,0)$};
            
            \draw (00) -- (01);
            \draw (00) -- (10);
            \draw (01) -- (11);
            \draw (10) -- (11);
            \draw (10) -- (20);
            \draw (20) -- (30);
            
            \node at (0, -0.8) {Order ideal $X$};
        \end{scope}

        \begin{scope}[xshift=5cm]
            \node (e)   at (0,0) {$e$};
            \node (t)   at (-1,1) {$t$};
            \node (s)   at (1,1) {$s$};
            \node (ts)  at (-1,2) {$ts$};
            \node (st)  at (1,2) {$st$};
            \node (sts) at (0,3) {$sts$};
            
            \draw (e) -- (t);
            \draw (e) -- (s);
            \draw (t) -- (ts);
            \draw (t) -- (st);
            \draw (s) -- (ts);
            \draw (s) -- (st);
            \draw (ts) -- (sts);
            \draw (st) -- (sts);
            
            \node at (0, -0.8) {Bruhat order of $S_3$};
        \end{scope}
        
    \end{tikzpicture}
    \caption{Hasse diagrams of the posets of Example \ref{esempio S3}.}
\end{figure}

\section{Non--existence of a Lehmer code for $F_4$}

In this section we show that a Coxeter system $(W,S)$ of type $F_4$ does not admit a Lehmer code by using the function \texttt{subgraph\_search\_iterator()} of SageMath. 
This result has already been stated in \cite{bishop}. We only add the fact that any bijective function $$L: \{x\in [1]_0\times [5]_0\times [7]_0\times [11]_0 : \rho(x)\leq 6\} \rightarrow \{w\in W : \ell(w)\leq 6\}$$ whose inverse is order--preserving cannot be extended to a bijective function $$\hat{L}: \{x\in [1]_0\times [5]_0\times [7]_0\times [11]_0 : \rho(x)\leq 7\} \rightarrow \{w\in W : \ell(w)\leq 7\}$$ whose inverse is order--preserving.

$\,$

 Let $C:=[1]_0\times [5]_0 \times [7]_0 \times [11]_0$. Then, if $\rho(x):=\sum_{i=1}^4x_i$, we have that
$$\sum\limits_{x \in C}q^{\rho(x)}=\sum\limits_{w \in F_4}q^{\ell(w)}=[2]_q[6]_q[8]_q[12]_q.$$
For $k\in \N$, let us define the sets 
\begin{eqnarray*}
    C(k)&:=&\{x \in C : \rho(x)\leq  k\}, \\
    F_4(k)&:=&\{w \in F_4: \ell(w)\leq k\}.
\end{eqnarray*}
Clearly $\sum_{x \in C(k)}q^{\rho(x)}=\sum_{w \in F_4(k)}q^{\ell(w)}$, for all $k\in \N$.

$\,$

Using the \texttt{subgraph\_search\_iterator()} function, we find $264$ distinct injective poset morphisms $L^{-1}:C(6) \rightarrow F_4(6)$. This computation requires about $18$ hours to be performed on our domestic devices.

$\,$

Having constructed a list $E$ of bijections $L: F_4(6) \rightarrow C(6)$ whose inverse functions are order preserving, 
we consider the set $Z$ of coatoms of the order ideal of $C$ with maximum $x$, for every element $x$ such that $\rho(x)=7$.

Since we want to construct an injective poset morphism $C \hookrightarrow W$, every cover relation in $C$ has to correspond to a cover relation in $W$. For every set $z \in Z$, the join of its elements lies in $\{x\in C: \rho(x)=7\}$. Then an embedding $f: C(6) \rightarrow F_4(6)$ can be extended to an embedding $f: C(7) \rightarrow F_4(7)$ only if there exists an element $w \in F_4(7)$ such that $f(c) \leq w$ for all $c \in z$, for every $z\in Z$. If such an element does not exist for some $z\in Z$, for every $f \in \mathrm{E}$, then there is no Lehmer code for $F_4$. This is in fact the case.

Hence, we have proved computationally the following theorem.

\begin{thm} Let $(W,S)$ be a Coxeter system of type $F_4$. Then there does not exist any
    bijective function $$\hat{L}: \{x\in [1]_0\times [5]_0\times [7]_0\times [11]_0 : \rho(x)\leq 7\} \rightarrow \{w\in W : \ell(w)\leq 7\}$$ whose inverse is order--preserving.
\end{thm}

The following question remains open.

\begin{que}
    Does $F_4$ admit a Lehmer code in the sense of Definition \ref{def codici di L 2}?
\end{que}

\section{Weak Lehmer codes} \label{sezione weak}

In this section we introduce the notion of \emph{weak Lehmer code} for a Coxeter group, extending the one of Lehmer code of a finite Coxeter group introduced in \cite{BS}. In the previous section we have shown that a Coxeter group of type $F_4$ does not admit a Lehmer code. As proved in \cite{BS}, a Lehmer code provides uniformly, for every lower Bruhat interval, a multicomplex whose $f$-polynomial equals the rank--generating function of the interval. By \cite[Theorem E]{BE} it is known that the ranks of a lower Bruhat interval in type $F_4$ form $M$-sequences, i.e. $f$-vectors of multicomplexes. By constructing an explicit weak Lehmer code, we obtain an algorithm to realize a multicomplex whose $f$-vector is such an $M$-sequence, for every lower Bruhat interval in type $F_4$.    

\vspace{0.3em}


\begin{dfn} \label{definizione weak}
    Let $(W,S)$ be a Coxeter system such that $k:=|S|<\infty$. We say that a finite set $L^W:=\{L_1,L_2,\ldots,L_h\}$ is a \emph{weak Lehmer code} for $(W,S)$ if:
    \begin{enumerate}
        \item[1)] $L_i: W \rightarrow \N^k$ is a rank--preserving injective function, for all $i \in [h]$;
        \item[2)] if $|\max\{L_i(v): v\leq w\}|=1$ then $\{L_i(v): v\leq w\}$ is a multicomplex, for every $w\in W$, $i\in [h]$;
        \item[3)] for every element $w\in W$ there exists $i\in [h]$ and $\phi \in \aut(W,\leq)$ such that the set $\{(L_i\circ \phi)(v): v\leq w\}$ is a multicomplex.
    \end{enumerate}
\end{dfn} 

By definition, if a Coxeter system admits a weak Lehmer code, then the sequence of ranks of any lower Bruhat interval is an $M$-sequence.

\vspace{0.3em}

\begin{ex}
 If $L$ is a Lehmer code, as in Definitions \ref{def codici di L} and \ref{def codici di L 2}, then $\{L\}$ is a weak Lehmer code. For a Lehmer code $L$, the set $\{L(v): v\leq w\}$ is a multicomplex, for every $w\in W$. See the proof of \cite[Proposition 4.3]{BS}.    
\end{ex}

\begin{ex}
 By the argument in the proof of \cite[Proposition 6.1]{BE}, a Coxeter system of type $\widetilde{C}_2$ does not admit any weak Lehmer code.
\end{ex}

\begin{ex}
Let $(W,S)$ be a universal Coxeter system with set of generators $S=\{r,s,t\}$; the group is defined by the relations $r^2=s^2=t^2=e$. By \cite[Proposition 2.2.9]{BB} there exists and element $w \in W$ such that $v<w$ for all $v \in W$ with $\ell(v)=3$. Hence the Poincaré polynomial, modulo $q^4$, of $[e,w]$ is $1+3q+6q^2+12q^3$. One can check, using the numerical characterization of $M$-sequences (see e.g. \cite[Chapter II]{StaCCA}) or directly by Definition \ref{def msequenza}, that $(1,3,6,12)$ is not an $M$-sequence. This implies that a universal Coxeter system $(W,S)$ with $|S|\geq 3$ does not admit any weak Lehmer code.
\end{ex}


The automorphism group $\aut(F_4,\leq)$ is isomorphic (see \cite[Corollary 2.3.6]{BB}) to the abelian group $\Z_2 \times \Z_2$ and its elements are $\{\id, \iota , \psi , \iota \circ \psi\}$, where:
\begin{itemize}
    \item $\iota: F_4 \rightarrow F_4$ is defined by the assignment $w \mapsto w^{-1}$, and 
    \item $\psi: F_4 \rightarrow F_4$ is defined by the assignment $$s_{i_1}s_{i_2}\cdots s_{i_k} \mapsto s_{\sigma(i_1)}s_{\sigma(i_2)}\cdots s_{\sigma(i_k)},$$ \noindent where $\sigma(i)=4-i+1$, for all $i\in [4]$. 
\end{itemize} 

\begin{figure}  
\begin{center}\begin{tikzcd}
s_1 \arrow[r, dash] & s_2 \arrow[r, dash] {r}{4} & s_3 \arrow[r, dash] &  s_4 
\end{tikzcd}\caption{Coxeter graph of $F_4$.}  \label{Coxeter graph} \end{center} 
\end{figure}
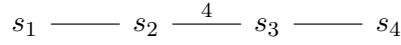

We now construct a weak Lehmer code for type $F_4$. Consider the following saturated chains of the Bruhat order:
   \begin{eqnarray*} X_1:= e \vartriangleleft s_2 \end{eqnarray*}
 \begin{eqnarray*}
     X_2:=  e \vartriangleleft s_3 \vartriangleleft s_3s_2 \vartriangleleft s_3s_2s_3
 \end{eqnarray*}
 \begin{eqnarray*}
X_3:= e \vartriangleleft s_1 \vartriangleleft s_2s_1 \vartriangleleft s_3s_2s_1 \vartriangleleft s_2s_3s_2s_1 \vartriangleleft s_1s_2s_3s_2s_1
 \end{eqnarray*}
\begin{eqnarray*}
        Y_1:= e &\vartriangleleft& s_4 \vartriangleleft s_4s_3 \vartriangleleft s_4s_3s_2 \vartriangleleft s_4s_3s_2s_1 \vartriangleleft s_4s_3s_2s_3s_1 \\ &\vartriangleleft& s_4s_ 3s_ 2s_ 3s_ 1s_ 2 \vartriangleleft
    	s_4s_ 3s_ 2s_ 3s_ 1s_ 2s_ 3 \vartriangleleft s_4s_ 3s_ 2s_ 3s_ 1s_ 2s_ 3s_ 4 \\ &\vartriangleleft& s_4s_ 3s_ 2s_ 3s_ 1s_ 2s_ 3s_ 4s_ 3 \vartriangleleft
    	s_4s_ 3s_ 2s_ 3s_ 1s_ 2s_ 3s_ 4s_ 3s_ 2 \\ &\vartriangleleft& s_4s_ 3s_ 2s_ 3s_ 1s_ 2s_ 3s_ 4s_ 3s_ 2s_ 3 \vartriangleleft s_4s_ 3s_ 2s_ 3s_ 1s_ 2s_ 3s_ 4s_ 3s_ 2s_ 3s_ 1 \\ &\vartriangleleft&
    	s_4s_ 3s_ 2s_ 3s_ 1s_ 2s_ 3s_ 4s_ 3s_ 2s_ 3s_ 1s_ 2 \vartriangleleft
    	s_4s_ 3s_ 2s_ 3s_ 1s_ 2s_ 3s_ 4s_ 3s_ 2s_ 3s_ 1s_ 2s_ 3 \\ &\vartriangleleft & s_4s_ 3s_ 2s_ 3s_ 1s_ 2s_ 3s_ 4s_ 3s_ 2s_ 3s_ 1s_ 2s_ 3s_ 4
    \end{eqnarray*}
 \begin{eqnarray*}
        Y_2:= s_4s_3s_2s_3 &\vartriangleleft& s_4s_3s_2s_3s_4 \vartriangleleft s_4s_3s_2s_3s_4s_1  
        \vartriangleleft s_4s_3s_2s_3s_4s_1s_2 \\ &\vartriangleleft& s_4s_3s_2s_3s_4 s_1s_2s_3 \vartriangleleft s_4s_3s_2s_3s_4s_1s_2s_3s_2 \\&\vartriangleleft&
         s_4s_3s_2s_3s_4s_1s_2s_3s_2s_1 \vartriangleleft s_4s_3s_2s_3s_1s_2s_3s_4s_3s_2s_1
    \end{eqnarray*}
    
$\,$

We have that $X_3X_1X_2= (F_4)_{\{s_1,s_2, s_3\}}$ is a parabolic subgroup and a Coxeter system of type $B_3$, and ${^{\{s_1,s_2, s_3\}}F_4}=Y_1 \uplus Y_2$. Then $$F_4=X_3X_1X_2Y_1 \, \uplus \, X_3X_1X_2Y_2.$$

\vspace{0.3em}

We define a function
$L_1: F_4 \rightarrow \N^4$ by setting
$$
 L_1(x_3x_1x_2u) =  \left\{
  \begin{array}{ll}
    (\ell(x_1),\ell(x_3),\ell(x_2),\ell(u)), & \hbox{if $u\in Y_1$;} \\
   (\ell(x_1),\ell(x_3),\ell(u),\ell(x_2)), & \hbox{if $u\in Y_2$,}
  \end{array}
\right.$$  for all $x_1\in X_1$, $x_2 \in X_2$, $x_3 \in X_3$, and $u\in Y_1 \cup Y_2$. Notice that
\begin{eqnarray*}
    \imm( L_1)&=& [1]_0\times [5]_0 \times [3]_0 \times [15]_0 \, \cup \, [1]_0\times [5]_0 \times [11]_0 \times [3]_0
    \\ &\subseteq& [1]_0\times [5]_0 \times [11]_0 \times [15]_0.
\end{eqnarray*}
We define a second function
$L_2: F_4 \rightarrow \N^4$ by setting
$$
 L_2(x_3x_1x_2u) =  \left\{
  \begin{array}{ll}
    (\ell(x_1),\ell(x_3),\ell(x_2),\ell(u)), & \hbox{if $u\in Y_1$;} \\
   (\ell(x_1),\ell(x_3),\ell(x_2)+4,\ell(u)-4), & \hbox{if $u\in Y_2$,}
  \end{array}
\right.,$$  for all $x_1\in X_1$, $x_2 \in X_2$, $x_3 \in X_3$, and $u\in Y_1 \cup Y_2$. 
The image of the function $L_2$ is 
\begin{eqnarray*}
    \imm( L_2)&=& [1]_0\times [5]_0 \times [3]_0 \times [15]_0 \, \cup \, [1]_0\times [5]_0 \times [7]_0 \times [7]_0
    \\ &\subseteq& [1]_0\times [5]_0 \times [11]_0 \times [15]_0.
\end{eqnarray*}

\begin{thm}
    The set $\{L_1,L_2\}$ is a weak Lehmer code for type $F_4$.
\end{thm}
\begin{proof}
The fact that $L_1$ and $L_2$ are rank--preserving follows by the factorization \eqref{fattorizzazione} (see also \cite[Section 3.4]{BB}). We can check computationally conditions $2)$ and $3)$ of Definition \ref{definizione weak}.
\end{proof}

We have now an algorithm to associate, in a global way, to any lower Bruhat interval of $F_4$ an order ideal of $\N^4$, i.e. a multicomplex, whose rank--generating function equals the one of the interval. 
Most of the multicomplexes associated with lower Bruhat intervals in $F_4$ are obtained by using the function $L_1$ composed with the automorphisms of the Bruhat order. 

By the general construction explained in \cite[Section 2]{BS}, we have an algorithm that associates to any lower Bruhat interval of $F_4$ a balanced vertex--decomposable simplicial complex whose $h$-polynomial equals the rank--generating function of the interval.

\begin{ex} Let $w:=s_4s_3s_1s_2s_3s_4s_2s_1 \in F_4$. The element $w$ is a fixed point of $\aut(F_4,\leq)$. We have that $L_1(w)=(1,1,6,0)$ and a multicomplex whose rank--generating function equals the one of the Bruhat interval $[e,w]$ is the order ideal $\{L_1(v):v\leq w\}$ with maxima
\begin{eqnarray*}
    \max\{L_1(v):v\leq w\} &=& \{(0,5,1,1),(1,1,0,5),(1,1,1,4), \\ && (1,1,3,1), (1,1,6,0),(1,2,1,2),(1,3,1,1)\}.
\end{eqnarray*}
The Lehmer complex of $[e,w]$ is a simplicial complex having $|[e,w]|=100$ facets.
\end{ex}

As in \cite[Section 6]{BS}, we define the set of $L$-principal and $L$-unimodal elements, for a rank--preserving injective function $L: W \rightarrow \N^k$.

\begin{dfn} \label{def principali} Let $(W,S)$ be a Coxeter system of finite rank  $k:=|S|$ with a weak Lehmer code $L^W$, and let $L\in L^W$. An element $w\in W$ is \emph{$L$-principal} if $|\max\{L(v):v\leq w\}|=1$. 
\end{dfn} We denote by $\Pr(L)$ the set of $L$-principal elements in $W$. For $w\in W$, we have defined $h_w:=\sum_{v\leq w}q^{\ell(v)}$. It is clear from Definitions \ref{def codici di L} and \ref{def principali} that 
\begin{eqnarray*}
    w\in \Pr(L) \, &\Leftrightarrow&  \,  \{L(v): v\leq w\} =  \{x \in \N^k : x\leq L(w)\}\\ &\Rightarrow & \,  h_w=\prod_{i=1}^k[L(w)_i+1]_q \,  .
\end{eqnarray*}

\begin{prop} \label{prop reticolo}
    Let $L^W$ be a weak Lehmer code for $(W,S)$ and $L\in L^W$. Then
    $(\Pr(L), \leq)$ is a meet-semilattice and, as posets, 
    $$(\Pr(L), \leq) \simeq (\{L(w): w\in \Pr(L)\},\leq).$$
\end{prop}

\begin{proof} Let $L\in L^W$ and $a,b\in  \Pr(L)$. Consider the element $m := a \wedge b$ in $\N^k$, i.e. 
$m=(\min\{a_1,b_1\},\ldots,\min\{a_k,b_k\})$; then $L^{-1}(m)\leq a$ and $L^{-1}(m)\leq a$ because $a$ and $b$ are $L$-principals. Let $v\in W$ be such that $v\leq L^{-1}(m)$; then $v\leq a$ and $v\leq b$, so $L(v) \leq L(a)$ and $L(v) \leq L(b)$, i.e. $L(v)\leq m$ and $L^{-1}(m)\in \Pr(L)$.

Since $L$ is a bijection, we have that its restriction to $\Pr(L)$ is a bijection. If $u\leq v$ in $\Pr(L)$, then $L(u)\leq L(v)$, by Definition \ref{def principali}. If $x\leq y$ in $\{L(w): w\in \Pr(L)\}$ then $L^{-1}(x)\leq L^{-1}(y)$ by condition $2)$ of Definition \ref{definizione weak}.
\end{proof}




For $w \in \Pr(L)$, define the orbit $O_w:=\{v\in \Pr(L): h_v=h_w\}$. 
\begin{dfn} \label{def unimodali}
   Let $(W,S)$ be a Coxeter system of finite rank with a weak Lehmer code $L^W$, and let $L\in L^W$.  An element $w\in W$ is \emph{$L$-unimodal} if $L(w)=\min_{\leq_{lex}}\{L(v):v \in O_w\}$.
\end{dfn}
We denote by $\U(L)$ the set of $L$-unimodal elements. The term \emph{unimodal} comes from the type $A$ case and the Lehmer code $L_n$ considered in \cite{BS}; in fact, in this case, $U(L_n)$ is the set of unimodal permutations in the symmetric group $S_n$, see \cite[Section 7]{BS}. By Proposition \ref{prop reticolo}, because $(\U(L), \leq)$ is an induced subposet of $(\Pr(L),\leq)$, we decuce the following result.
\begin{prop} \label{prop unimodali} Let $(W,S)$ be a Coxeter system of finite rank with a weak Lehmer code $L^W$, and let $L\in L^W$. Then
    the Bruhat order on $\U(L)$ is isomorphic to the componentwise order on $\{L(w):w \in \U(L)\}$.
\end{prop}
Let us define $$\Pal(F_4):=\left\{h_w: w\in F_4, \, h_w=q^{\ell(w)}h_w(q^{-1})\right\}.$$ In other words, $\Pal(F_4)$ is the set of Poincaré polynomials for the cohomology of rationally smooth Schubert varieties in type $F_4$. By Definition \ref{def unimodali}, the assignment $w\mapsto \prod_{i=1}^k [L_1(w)_i+1]_q$ defines an injective function $\U(L_1) \rightarrow \Pal(F_4)$.
In the following proposition we obtain the analogue of \cite[Th. 7.11 and Ex. 6.8]{BS} for type $F_4$. 
\begin{prop}
The sets $\Pr(L_1) \cup \{w_0\}$ and $\U(L_1) \cup \{w_0\}$ with the induced Bruhat order are lattices and 
\begin{eqnarray*}
   \Pal(F_4) &=&\{h_w: w \in \U(L_1)\cup \{w_0\} \} \\
            &=& \left\{\prod_{i=1}^4[L_1(w)_i+1]_q: w \in \U(L_1) \right\}  \cup \, \left\{[2]_q[6]_q[8]_q[12]_q\right\}.
 \end{eqnarray*}
\end{prop}
\begin{proof}
    The lattice property  of  $\Pr(L_1) \cup \{w_0\}$ follows by Proposition \ref{prop reticolo} and the existence of the maximum $w_0$.
We now construct the orbit set $\mathrm{O}$ of the elements $w\in \Pr(L_1)$.

Now, for each orbit, we choose a representative by taking the minimum in the lexicographic order.

We now check that the cardinality of $\Pal(F_4)$ is equal to $|\U(L_1)|+1$.

\end{proof}

\begin{oss}
    One can check that $|U(L_2)|<|\Pal(F_4)|-1$ and that the set $U(L_2)$ with the Bruhat order is not a lattice.
\end{oss}

\begin{ex}
    Let $c:=s_1s_2s_3s_4$ and $w:=cs_1$. We have that $L_1(w)=(1,2,1,1)$, 
    $h_w=[2]_q^3[3]_q$ and $w\in \Pr(L_1)$. Moreover $O_w=\{cs_1,cs_2,cs_3\}$ and $cs_3\in \U(L_1)$; in fact, $L_1(cs_2)=(1,1,2,1)$ and $L_1(cs_3)=(1,1,1,2)$.
\end{ex}

In Figure \ref{Hasse0} we depict the Hasse diagram of the sorted image,  under the function $L_1$,  of $\U(L_1)\cup\{w_0\}$, ordered componentwise. This diagram provides the complete list of palindromic Poincaré polynomials of rationally smooth Schubert varieties for type $F_4$. For any $w\in \U(L_1)$, a tuple $(x_1,x_2,x_3,x_4)=\sort(L_1(w))$ corresponds to the polynomial $h_w=\prod_{i=1}^4[x_i+1]_q$.
The lattice of Figure \ref{Hasse0} is a sublattice of $[1]_0\times [5]_0\times [7]_0 \times [11]_0$ and this explains its distributivity. 

In Figure \ref{Hasse1} we depict the Hasse diagram of $\U(L_1)\cup\{w_0\}$ with the induced Bruhat order. By Proposition \ref{prop unimodali}, such poset coincides with the image of $\U(L_1)\cup\{w_0\}$ under the function $L_1$, ordered componentwise. Notice that the lattice of Figure \ref{Hasse1} is not a sublattice of $[1]_0\times [5]_0\times [7]_0 \times [11]_0$; in fact
$(0,0,1,1) \wedge (1,2,2,0) = (0,0,0,0) \neq (0,0,1,0)$.

$\,$


{\bf Funding:}
The first author was supported by the grant PRIN
2022K48YYP Unirationality, Hilbert schemes, and singularities.

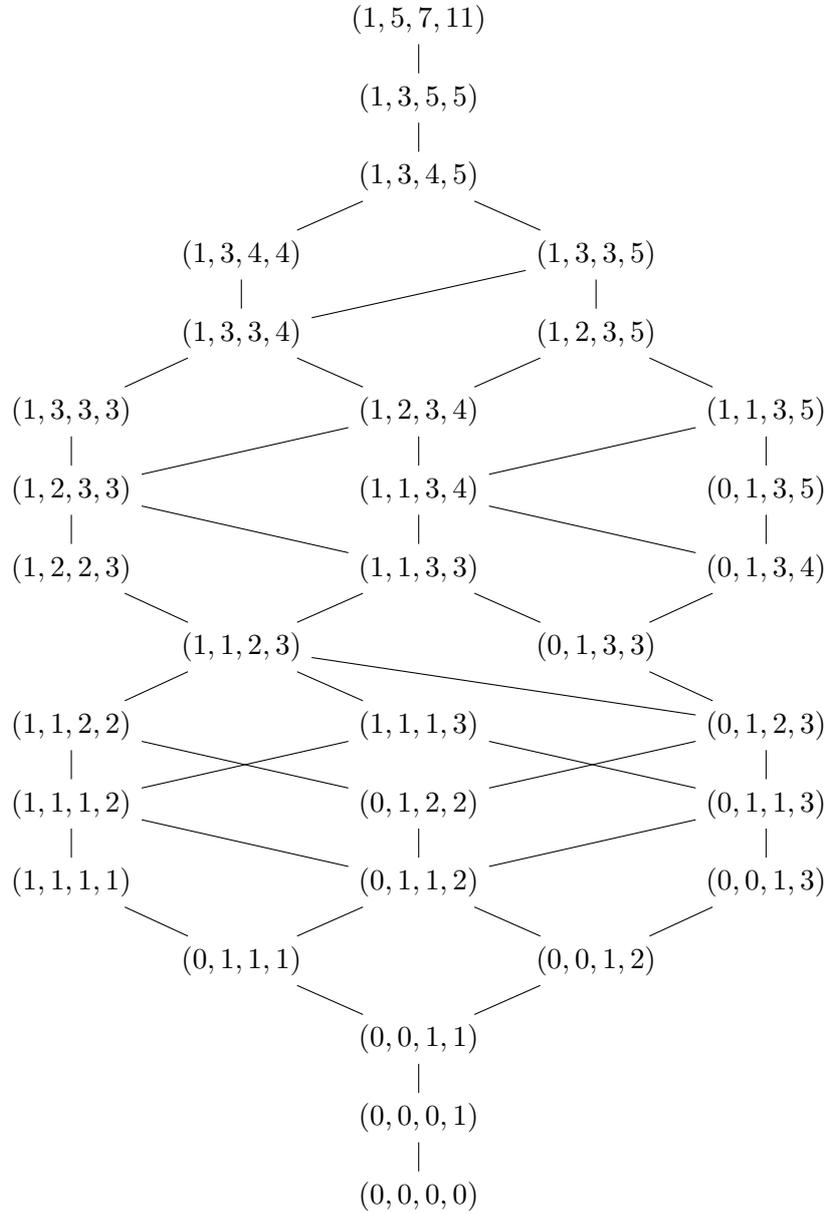
\begin{figure} 
\captionsetup{justification=centering,margin=2cm}
\begin{center}\begin{tikzpicture}
\matrix (a) [matrix of math nodes, column sep=0.4cm, row sep=0.4cm]{
            &           & (1,5,7,11)&           & \\
            &           & (1,3,5,5) &           & \\
            &           & (1,3,4,5) &           & \\
            & (1,3,4,4) &           & (1,3,3,5) & \\
            & (1,3,3,4) &           & (1,2,3,5) & \\
  (1,3,3,3) &           & (1,2,3,4) &           & (1,1,3,5)  \\
  (1,2,3,3) &           & (1,1,3,4) &           & (0,1,3,5)\\
  (1,2,2,3) &           & (1,1,3,3) &           & (0,1,3,4)\\
            & (1,1,2,3) &           & (0,1,3,3) & \\
  (1,1,2,2) &           & (1,1,1,3) &           & (0,1,2,3) \\
  (1,1,1,2) &           & (0,1,2,2) &           & (0,1,1,3) \\
  (1,1,1,1) &           & (0,1,1,2) &           & (0,0,1,3)\\
            & (0,1,1,1) &           & (0,0,1,2) & \\
            &           & (0,0,1,1) &           & \\
            &           & (0,0,0,1) &           & \\
            &           & (0,0,0,0) &           & \\};

\foreach \i/\j in {1-3/2-3, 2-3/3-3, 3-3/4-2, 3-3/4-4, 4-2/5-2, 4-4/5-2,
4-4/5-4, 5-2/6-1, 5-2/6-3, 5-4/6-3, 5-4/6-5, 6-1/7-1, 6-3/7-1, 6-3/7-3, 6-5/7-3, 6-5/7-5,
7-1/8-1, 7-1/8-3, 7-3/8-3, 7-3/8-5, 7-5/8-5, 8-1/9-2, 8-3/9-2, 8-3/9-4, 8-5/9-4,
9-2/10-1, 9-2/10-3, 9-2/10-5, 9-4/10-5, 10-1/11-1, 10-1/11-3, 10-3/11-1, 10-3/11-5, 
10-5/11-3, 10-5/11-5, 11-1/12-1, 11-1/12-3, 11-3/12-3, 11-5/12-3, 11-5/12-5,
12-1/13-2, 12-3/13-2, 12-3/13-4, 12-5/13-4, 13-2/14-3, 13-4/14-3, 14-3/15-3, 15-3/16-3}
    \draw (a-\i) -- (a-\j);
\end{tikzpicture} 
\caption{Hasse diagram of the distributive lattice \\ 
 $\{\sort(L_1(w)):w \in \U(L_1)\} \cup \{(1,5,7,11)\}$}  \label{Hasse0} \end{center} \end{figure}

\begin{figure} 
\captionsetup{justification=centering,margin=0cm}
\begin{center}
\begin{tikzpicture}
\matrix (a) [matrix of math nodes, column sep=0.4cm, row sep=0.4cm]{
           &           &           & (1,5,7,11)&           &            \\
           &           &           & (1,5,5,3) &           &            \\
           &           &           & (1,4,5,3) &           &            \\
           &           & (1,4,4,3) &           & (1,3,5,3) &            \\   
           &           & (1,3,4,3) &           & (1,2,5,3) &            \\    
           & (1,3,3,3) &           & (1,2,4,3) &           & (1,1,5,3)  \\   
           & (1,2,3,3) &           & (1,1,4,3) &           & (1,0,5,3)  \\ 
           & (1,2,3,2) &           & (1,1,3,3) &           & (1,0,4,3)  \\
           &           &           & (1,1,3,2) &           & (1,0,3,3)  \\ 
           & (1,2,1,2) &           & (1,1,3,1) &           & (1,0,3,2)  \\ 
           & (1,1,1,2) &           & (1,2,2,0) &           & (1,0,3,1)  \\ 
           & (0,1,1,2) & (1,1,1,1) &           &           & (1,0,3,0)  \\     
           & (0,0,1,2) & (0,1,1,1) &           &           &            \\    
           &           & (0,0,1,1) &           &           &            \\
           &           & (0,0,0,1) &           &           &            \\
           &           &           & (0,0,0,0) &           &            \\};

\foreach \i/\j in {1-4/2-4, 2-4/3-4, 3-4/4-3, 3-4/4-5, 4-3/5-3, 4-5/5-3, 4-5/5-5, 5-3/6-2, 5-3/6-4, 5-5/6-4, 5-5/6-6, 6-2/7-2, 6-4/7-2, 6-4/7-4, 6-6/7-4, 6-6/7-6, 7-2/8-2, 7-2/8-4, 7-4/8-4, 7-4/8-6, 7-6/8-6, 8-2/11-4, 8-2/10-2, 8-2/9-4, 8-4/9-4, 8-4/9-6, 8-6/9-6, 9-4/10-4, 9-4/11-2, 9-4/10-6, 9-6/10-6, 10-2/11-2, 10-4/12-3, 10-6/11-6, 11-4/16-4, 11-2/12-3, 11-2/12-2, 11-6/12-6, 12-3/13-3, 12-2/13-3, 12-2/13-2, 12-6/16-4, 13-3/14-3, 13-2/14-3, 14-3/15-3, 15-3/16-4} 
\draw (a-\i) -- (a-\j);
\end{tikzpicture} 
\end{center}
\caption{Hasse diagram of the lattice 
$\{L_1(w):w \in \U(L_1)\} \cup \{(1,5,7,11)\}$. It is also the Hasse diagram of the Bruhat order on $\U(L_1) \cup \{w_0\}$.}  \label{Hasse1} 
\end{figure}
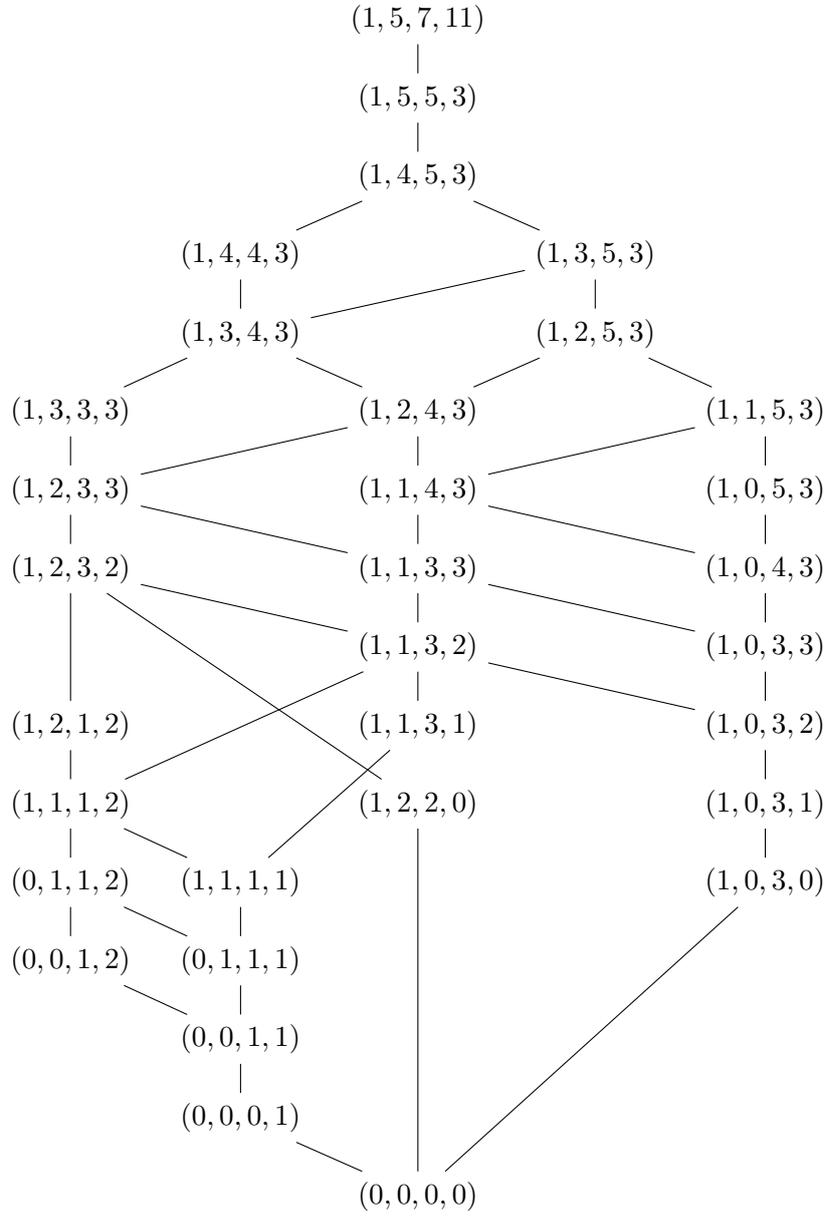   



\begin{thebibliography}{7}

\bibitem{Billey1}
S. C. Billey, C. K. Fan, and J. Losonczy, {\em The parabolic map}, Journal of Algebra 214.1, 1-7 (1999).

\bibitem{billei} 
S. C. Billey, {\em Pattern avoidance and rational smoothness of Schubert varieties}, Advances in Mathematics 139.1, 141-156 (1998).

\bibitem{bishop} 
A. Bishop, E. Milićević, and A. Thomas, {\em Trivial Kazhdan-Lusztig polynomials and cubulation of the Bruhat graph}, arXiv preprint arXiv:2504.03046 (2025).

\bibitem{BB}
A. Bj\"{o}rner and F. Brenti, {\em Combinatorics of Coxeter Groups},
Graduate Texts in Mathematics, 231, Springer-Verlag, New York, 2005.

\bibitem{BE}
A.  Bj\"{o}rner and T. Ekedahl, {\em On the shape of Bruhat intervals}, Annals of mathematics, 799-817 (2009).

\bibitem{BS}
D. Bolognini and P. Sentinelli, {\em The Lehmer complex of a Bruhat interval}, arXiv preprint arXiv:2501.03037 (2025).

\bibitem{Carrell}
J. B. Carrell, {\em The Bruhat graph of a Coxeter group, a conjecture of Deodhar, and rational smoothness of Schubert varieties.} Proceedings of Symposia in Pure Mathematics. American Mathematical Society, 1994.



\bibitem{SageMath}
{The Sage Developers}.
\newblock {\em {S}ageMath, the {S}age {M}athematics {S}oftware {S}ystem({V}ersion 9.3)}.

\bibitem{StaCCA}
R. P. Stanley, {\em Combinatorics and Commutative Algebra}, 
Vol. 41. Springer Science \& Business Media (2007).



\end{thebibliography}
\end{document}